\documentclass[11pt]{amsart}
\usepackage[T1]{fontenc}
\usepackage[normalem]{ulem}
\usepackage{lmodern, amsfonts,amsmath,amstext,amsbsy,amssymb,
amsopn,amsthm,upref,eucal,mathptmx}

\usepackage{mathrsfs}

\RequirePackage{xcolor} % [dvipsnames]
\definecolor{halfgray}
{gray}{0.55}%chapter numbers will be semi
\definecolor{webgreen}
{rgb}{0,0.4,0}
\definecolor{webbrown}
{rgb}{.8,0.1,0.1}
\definecolor{red}
{rgb}{1,0,0}
\usepackage{microtype}

\newcommand \R {{ \mathbb R}}

\newcommand \Z {{ \mathbb Z}}
\newcommand \N {{ \mathbb N}}
\newcommand \T {{ \mathbb T}}
\newcommand \Q {{ \mathbb Q}}

\newtheorem{theorem}{Theorem}[section]
\newtheorem {lemma} [theorem]{Lemma}

\newtheorem{definition}[theorem]{Definition}
\newtheorem{conjecture}[theorem]{Conjecture}
\newtheorem{question}[theorem]{Question}

\title[Counterexamples to a rigidity conjecture ]%
{Counterexamples to a rigidity conjecture}

  \author{Giovanni Forni}
  \author{Adam Kanigowski}

\address{Department  of Mathematics\\
  University of Maryland \\
  College Park, MD USA}
  
\email
    {gforni@math.umd.edu}
\email{akanigow@umd.edu}

\date{\today}

\begin{document}

\def\echo#1{\relax}
    
\begin{abstract}
  \begin{sloppypar}
 We discuss several counterexamples to a rigidity conjecture of K.~Khanin \cite{Kha18}, which states that under some quantitative condition on non-existence of periodic orbits,  $C^0$ conjugacy implies $C^1$ (even $C^\infty$) conjugacy.  We construct examples of non-rigid diffeomorphisms on the $2$-torus, which  satisfy the assumptions of Khanin's   (but not of Krikorian's) conjecture.  We also construct examples of flows which are topologically conjugate, but not 
 $C^1$ conjugate in contradiction  to a natural generalization of the conjecture to flows. These latter examples are 
 based on results on solutions  of the cohomological equation and suggest that the structure of the space of invariant 
 distributions has to be taken into account in rigidity questions. 
   \end{sloppypar}
\end{abstract}

\date{\today}

 \maketitle
 \section{Introduction}  Several important results in low dimensional smooth dynamics concern the property of {\it geometric rigidity}
 of the dynamical system, that is, the property that all continuous conjugacies between the system and another system within
 a given class of systems are in fact $C^1$ or even $C^\infty$.

 The prototype of all  results of this kind is the famous theorem by M.~Herman which implies that almost all rotations of the circle are
 $C^\infty$-rigid,  in the sense that all continuous conjugacies with smooth circle diffeomorphisms are in fact $C^\infty$. This result was  later 
 strengthened  by J.~C.~Yoccoz to include all Diophantine rotations. 
 
 As reported in K. Khanin's ICM talk \cite{Kha18}, a generalization of the Herman--Yoccoz result to translations on higher dimensional tori has been conjectured by 
 R.~Krikorian: 
 
 \begin{conjecture} (R.~Krikorian)   Let $T$ be a $C^\infty$ diffeomorphism of $\T^d$. Assume that $T$ is topologically conjugate to a linear translation $T_\omega: {\bf x} \to
 {\bf x} + \omega (\text{\rm mod } 1)$ with a Diophantine rotation vector  $\omega= (\omega_1, \dots, \omega_d)$.  Then the conjugacy between $T$ and $T_{\omega}$ is $C^\infty$ smooth.
  \end{conjecture}
  
  K.~Khanin \cite{Kha18}, proposes a  generalization of the above conjecture to arbitrary manifolds. Indeed, he proposes the following {\it Rigidity conjecture} \cite{Kha18}.
  
  \smallskip
  {\it Let $M$ be a compact Riemannian manifold. Let $T_1$ and $T_2$ be $C^\infty$ diffeomorphisms of $M$.  Assume  that $T_1$ and $T_2$ satisfy the Diophantine
  property, namely there exist $\tau$, $C>0$ such that for every $x \in M$ and all $n\in \N$ we have}
  \begin{equation}
  \label{eq:Khanin_C}
  \text{dist} (T_1^n x, x) \geq  C n^{-\tau} \,, \qquad    \text{dist} (T_2^n x, x) \geq  C n^{-\tau}  \,.
  \end{equation}

  \begin{conjecture} (K.~Khanin, \cite{Kha18}) Suppose $T_1$, $T_2$ are topologically conjugate. Then the conjugacy is $C^\infty$ smooth.
  
  Probably it is enough to require the Diophantine condition for only one of the maps $T_1$ and $T_2$. 
 \end{conjecture}
 
 While we don't know whether Krikorian's conjecture holds, it is a rather simple task to exhibit a counterexample to the above conjecture
 as stated, even for translations of the $2$-dimensional torus (see \S \ref{sec:T2cex}). This is because the condition \eqref{eq:Khanin_C} is too weak in higher dimensions as it does not exclude Liouvillean factors.  Indeed, we have the following
 
\begin{theorem} 
 \label{thm:counterex}
 There exist $C^{\infty}$ diffeomorphisms $T_1$ and $T_2$ of $\T^2$, which satisfy  condition \eqref{eq:Khanin_C}
 and are $C^0$ conjugate, but not $C^1$ conjugate. 
 \end{theorem}
  While one may hope to rectify the conjecture by introducing an appropriate generalization of  the classical Diophantine condition to ensure absence of Liouvillean factors (at least), we construct examples that cast serious doubts on whether any such general conjecture (only assuming quantitative conditions on returns of orbits) can be true.

It is not just that the classical Diophantine condition is not simply about periodic orbits, or close returns, or even about invariant sets of lower codimension, it is also that the case of toral flows is very special since irrational  linear toral flows have no  {\it invariant distributions} other than the mean. 
 
 In this paper we construct examples of flows, based on the theory of cohomological equations for horocycle flows, translation flows, and nilflows, which suggest that 
 not only invariant distributions have to be taken into account, but in fact it is {\it the structure of the space of invariant distributions that determines
 whether the flow is geometrically rigid}. 
 
 \medskip
 
 The first set of examples is given by time-changes of homogeneous flows (horocycles flows and Heisenberg nilflows)  or translation flows, the second set of examples is similar,  given by skew-product flows over the above mentioned flows with circle fibers.
 
 We remark that if a  flow satisfies the Diophantine condition~\eqref{eq:Khanin_C} (adapted to continuous time $t>1$ instead of discrete time $n\in \N$), it follows that continuous time changes and skew product flows over the given flow also satisfy the condition.
 
 It is well-known that the Krikorian conjecture holds for time-changes and  skew-products of Diophantine toral linear flows with circle fibers, by the theory of cohomological 
 equations. For general manifolds (compact, or at least finite volume) we produce counterexamples to the continuous-time version of Khanin's conjecture, in the case of horocycle flows and Kochergin flows,  as well as results in agreement with  it in the case of nilflows.

 \begin{theorem}
 \label{thm:main1a} Horocycle flows on unit tangent bundles of compact (or finite-volume) hyperbolic surfaces or generic Kochergin
 locally Hamiltonian flows on surfaces of higher genus  are not $C^1$-rigid for time changes (outside of a subspace of finite codimension 
 in the space of pairs of time change functions), in the sense that there exist smooth time-changes which are $C^0$ but not $C^1$ conjugate.
 
In addition, there exist smooth time-changes of horocycle flows with time-$T$ maps which are $C^0$ conjugate but not $C^1$-conjugate.
  \end{theorem}

 \begin{theorem}
 \label{thm:main1b} Skew product flows  over horocycle flows on unit tangent bundles of compact hyperbolic surfaces or generic Kochergin
 locally Hamiltonian flows on surfaces of higher genus, defined by a cocycle with values in the circle, are in general not $C^1$-rigid (outside of a subspace of finite codimension 
 in the space of pairs of skew-product flows).  
  \end{theorem}
 
 In the opposite direction we prove the following partial results:

   \begin{theorem}
  \label{thm:main2a}
 Heisenberg nilflows are $C^\infty$-rigid within the class of smooth time-changes, in the sense that if a $C^\infty$ time-change of a Diophantine 
 Heisenberg nilflow is $C^0$ conjugate to the Heisenberg nilflow, then the conjugacy is $C^\infty$. 
 \end{theorem}

  \begin{theorem}
  \label{thm:main2b}
  Skew product flows over Diophantine Heisenberg nilflows are $C^\infty$-rigid within the class of skew-product flows, in the sense that if a $C^\infty$ skew-product flow 
  over a Diophantine Heisenberg nilflow, defined by a cocycle with values in the circle,
 is $C^0$ conjugate to a skew-product flow (with circle fibers) over the same Heisenberg nilflow, then the conjugacy is $C^\infty$. 
 \end{theorem}

 We remark that, as consequence of the fact that horocycle flow, translation flows, and Heisenberg nilflows are renormalized respectively by the geodesic
 flow, the Teichm\"uller flow (or the Rauzy--Veech induction) and the diagonal flow on $ASL(2, \Z) \backslash ASL(2, \R)$ (see \cite{FlaFo06}), it follows
 that the generic skew-product flows satisfy the analogs of Khanin's condition~\eqref{eq:Khanin_C} for flows. Similar constructions yield maps (instead of flows) 
 with similar properties (see for instance \cite{FFT16}  for results on the cohomological equation for horocycle maps).

In light of our result for skew-product flows over Heisenberg nilflows, and of the results \cite{FlaFo07} on the cohomological equations for general nilflows,
we propose the following conjecture
 \begin{conjecture} 
 Diophantine Heisenberg (or, more generally, step $2$) nilflows are $C^\infty$ geometrically rigid.
 \end{conjecture} 
 
 This conjecture has been proposed also by B.~Fayad in personal communications and lectures, and it is further supported by the results of D.~Damjanovic, B.~Fayad and
M.~Saprykina~\cite{DFS} on the rigidity of higher rank step $2$ nilpotent actions. 

A generalization of this conjecture to the higher step nilpotent case is problematic.  

Indeed, as remarked by B.~Fayad (see also \cite{FlaFo07})  in the higher step case the solution operator of the cohomological equation is not 
tame (in the sense of R.~Hamilton) and the lack of tameness is sufficient to make a KAM approach to the conjugacy problem difficult to envision. In addition, no renormalization scheme  is known for higher step nilflows. Even if the conjecture holds, in the higher step case it appears to be completely out of reach.

 \bigskip 
 A toy model for the geometric rigidity problem is perhaps given by the analogous question for the linearization of the conjugacy problem, which in turn can be reduced 
 (at least in the case of unipotent flows and nilflows) to the (scalar) cohomological equation. 
 
 Following A.~Katok \cite{Kat03},  a smooth dynamical system is called  {\it $C^k$-effective } if all measurable (or $C^0$) solutions of the (scalar) cohomological equation 
 are of class $C^k$.  The question of under what condition a diffeomorphism or a flow is $C^1$ or $C^\infty$-effective offers perhaps a guide about what to expect for the harder, 
 non-linear, geometric rigidity problem. 
 
 From this point of view, we immediately see that the case of linear flows on the torus is (at least conjecturally in general) very special.  A conjecture of Katok
 (see \cite{Kat03}, Conjecture 3.6)  states that  only diffeomorphisms smoothly conjugated to Diophantine translations on tori are {\it $C^\infty$- rigid (or cohomology-free)},  in the sense that the cohomological equation has a solution for every smooth functions up to an additive constant.  A similar conjecture can be formulated for flows. 
 
 Katok's conjecture for flows 
 is known up only to dimension $3$ \cite{Fo08}, \cite{Ko09}, \cite{Ma09} but has been proven in general for {\it homogeneous} flows and affine maps \cite{FFRH16}.  Indeed the latter paper proves that all non-toral homogenous dynamical systems have an infinite dimensional space of invariant distributions. 
 
 As suggested by the examples given below, whether a smooth dynamical system is geometrically rigid is related to the structure of the space of its invariant
 distributions. In the examples we understand,  invariant distributions of higher regularity (of order $1$) carry the non-trivial  (polynomial) deviation of ergodic averages, hence obstruct the existence of  continuous and even measurable solutions of the cohomological equation. 
 Thus, $C^k$-cocycle effectiveness, and perhaps $C^k$-rigidity, is
 related to the absence of invariant distributions of lower regularity which may obstruct the existence of $C^k$ solutions for $k\geq 1$. 
 
 Currently, the relation between invariant distribution and deviation of ergodic averages is understood via renormalization methods in a few
 classes of examples. 
 
 It is unclear how to generalize these results to systems with no known renormalization, other than by a scaling approach \cite{FlaFo14}, \cite{FFT16}, 
 by methods based on the intrinsic dynamical features of the system.
 
 \bigskip
 The paper is organized as follows. In \S \ref{sec:T2cex} we give a counterexample to Khanin's rigidity conjecture, which underlines that the Diophantine
 condition on tori is not simply a condition about periodic orbits.
 In \S  \ref{sec:time_skew} we relate conjugacies of time changes and skew-product flows and solutions of the cohomological equation for the base flow. 
 In \S \ref{sec:cohom} we derive converse effectiveness results on cohomological equations  for horocycle flows and translation flows, and an effectiveness result 
 for Heisenberg nilflows. Finally in \S \ref{sec:proofs} we give precise statements and proofs of our main results.

 \section{A counterexample on the $2$-torus}
 \label{sec:T2cex}
 In this section we prove Theorem ~\ref{thm:counterex}. Let $R_\beta:\T\to \T$ denote the rotation by $\beta\in \R$. For irrational $\beta$ let $(q^\beta_n)$ denote its sequence of denominators. Then we have $R_\beta^{q^\beta_n}\to \text{\rm Id}$ uniformly on $\T$. Let now $\beta$ be any (Liouville) irrational number satisfying the following property:
\begin{itemize}
\label{eq:beta}
\item  there exists a $C^\infty$, or even real-analytic, circle diffeomorphism $g:\T\ \to \T$ topologically conjugate to $R_\beta$ such that the conjugacy is not of class $C^1$.
\end{itemize}
Notice that the existence of such $\beta \in \R\setminus \Q$ follows by Arnol'd  \cite{Ar61} result on existence of real- analytic circle diffeomorphisms with irrational rotation number for which the
conjugacy to a rotation is not absolutely continuous (while it is continuous by Denjoy theorem).

For simplicity of notation, we will denote $(q_n):=(q^\beta_n)$. Then we have the following:

\begin{lemma} 
\label{lemma:notC1}

Let $\alpha$ be any Diophantine number such that there exists a diverging sequence $(n_k) \subset \N$ for which 
$$
(q_{n_k}+1)\alpha\to 0.
$$
Then
$T_1:= R_\alpha\times R_\beta$ is topologically conjugate to $T_2:= R_\alpha \times g$, but not $C^1$-conjugate.
\end{lemma}
\begin{proof} Notice that topological conjugacy follows immediately by the assumptions: if $H: \T \to \T$ is a homeomorphism conjugating $R_\beta$ and $g$, then 
$\text{ \rm Id} \times H$ is a homeomorphism conjugating $T_1= R_\alpha\times R_\beta$ and  $T_2= R_\alpha\times g$.  

To prove that there is no differentiable conjugacy, we argue by contradiction. Let us assume that there exists a $C^1$ diffeomorphism $F=(F_1,F_2):\T^2\to \T^2$ which conjugates the two above maps, i. e. 
$$
F\circ (R_{\alpha}\times R_\beta) =(R_\alpha\times g)\circ F\,.
$$

By looking only at the second coordinate, the above conjugacy identity then implies that
$$
F_2(x+\alpha,y+\beta)= g(F_2(x,y)) \,, \quad \text{ for all }  (x,y)\in \T^2\,.
$$
By iterating the above identity we derive that, for all $k\in \N$ and for all $(x,y)\in \T^2$,
$$
F_2(x+(q_{n_k}+1)\alpha,y+(q_{n_k}+1)\beta)=g^{q_{n_k}+1}(F_2(x,y)).
$$
Since by assumption $q_{n_k}\beta\to 0$, and so $R^{q_{n_k}}_\beta\to \text{\rm Id}$, and  $H\circ R_\beta \circ H^{-1}=g$, 
we have that 
$$
g^{q_{n_k}+1}=H\circ R^{q_{n_k}+1}_\beta \circ H^{-1}=H\circ R^{q_{n_k}}_\beta \circ H^{-1}\circ H\circ R_\beta \circ H^{-1}\to g\,.
$$
By taking the limit as $n_k\to \infty$ in the above conjugation identity, we derive that  
$$
F_2(x,y+\beta)=g(F_2(x,y)) \,, \quad \text{ for all }  (x,y)\in \T^2\,,
$$
but then, if  for any fixed $x\in \T$ we let $\tilde{H}(\cdot):=F_2(x,\cdot)$, we get the conjugation identity
$$
\tilde{H}\circ R_\beta=g\circ\tilde{H}.
$$
Moreover, $\tilde H$ is of class $C^1$ (since $F$ was assumed to be of class $C^1$). This however contradicts the assumption on the diffeomorphism 
$g: \T \to \T$ that its conjugacy $H$ to the rotation $R_\beta$ is not of class $C^1$. The proof of the lemma is finished.
\end{proof}

Finally, it is enough to show that the set of Diophantine numbers satisfying $(q_{n_k}+1)\alpha\to 0$ for some subsequence $(n_k)\subset \N$ is non-empty:
\begin{lemma} 
\label{lemma:full_meas}

For any diverging sequence ${\bf r}:=(r_n)$  of positive integers, the set
$$
A({\bf r}):=\{\alpha \in \T  \;:\; r_{n_k}\alpha\to 0\text{ for some diverging sequence } (n_k)\subset \N\},
$$
has  full Lebesgue measure.
\end{lemma}
\begin{proof}  By replacing the sequence ${\bf r}$ by a subsequence, it is not restrictive to assume that the sequence $(r_{n+1}/r_n)$ is divergent (and increasing).
Let $(\epsilon_n) \subset (0,1)$ be a sequence decreasing to $0$ in $\R$ such that 
and 
$$
\frac{r_{n+1}}{r_n}  \epsilon_n   \geq 1\,, \quad   \text{ for all }  n\in \N\,,
$$
and such that the series $\sum \epsilon_n$ is divergent. For all $n\in \N$, the set
$$
A_n=\{\alpha \in \T :   \| r_{n}\alpha\|_\Z \leq \epsilon_n\}
$$
is equal to the union of $r_n$ disjoint intervals of length $2 \epsilon_n/r_n$, hence in particular has Lebesgue measure $2\epsilon_n$. 
It follows that, for every $m>n$  and for every maximal interval $I_n \subset  A_n$, since $2r_m \epsilon_n/r_n \geq 2$, then
$$
\text{\rm Leb} (I_n \cap A_m) \leq  2 \epsilon_m   \text{\rm Leb} (I_n) +  \frac{\epsilon_m}{r_m}  =   \text{\rm Leb}(A_m)\left(  \text{\rm Leb} (I_n) + \frac{1}{r_m} \right)   \,,
$$
hence the sets $A_m$ and $A_n$ are pairwise independent, in the sense that
$$
\text{\rm Leb} (A_n \cap A_m) \leq     \text{\rm Leb} (A_m)\left(  \text{\rm Leb}(A_n)  +  \frac{r_n} {r_m}  \right) \,.
$$
By  replacing the sequence ${\bf r}$ by a subsequence, it is then possible to choose the sequence $(\epsilon_n)$ with the property that
$$
 \frac{   \sum_{k=1}^n  \frac{1}{r_k}    \sum_{h=1}^{k-1}  r_h \text{\rm Leb}( A_h ) }{ (\sum_{k=1}^n  \text{\rm Leb}( A_k ) )^2  }   \to  0.
$$
 The results then follows from the Kochen-Stone inequality, since
$$
\text{\rm Leb}( A({\bf r}  ) ) = \text{\rm Leb}(\liminf_{n\to \infty} A_n)  \geq  \limsup_{n\to \infty} \frac{ (\sum_{k=1}^n  \text{\rm Leb}( A_k ) )^2  }{ \sum_{k,h=1}^n  \text{\rm Leb}( A_k \cap A_h ) } \geq 1\,.
$$
\end{proof}

\begin{proof} [Proof of Theorem \ref{thm:counterex}]
Let $\beta \in \R\setminus \Q$ be a Liouvillean irrational number such that there exists a real-analytic diffeomorphism of the circle $g$, of rotation number 
$\rho(g)=\beta$,  which is not absolutely continuous conjugate to the rotation $R_\beta$.  Let $(q_n^\beta)$ denothe the sequence of the the denominators
of the continued fraction expansion of $\beta$, and let  ${\bf r} = (r_n)$  denote the sequence
$$
r_n :=    q_n^\beta +1 \,, \quad \text{ for all } n\in \N\,. 
$$
Since $\bf r$ is divergent, by Lemma~\ref{lemma:full_meas} there exists a Diophantine number $\alpha\in \R\setminus \Q$
(of any exponent $\tau >1$) such that there exists a sequence $(n_k) \subset \N$ with
$$
(q^\beta_{n_k} +1) \alpha \to  0 \,.
$$
Let then $T_1:= R_\alpha \times R_\beta$  and $T_2= R_\alpha \times g$ on $\T^2$.  By Lemma \ref{lemma:notC1} the smooth (real-analytic) diffeomorphisms
$T_1$ and $T_2$ are topologically conjugate, but not $C^1$  conjugate (in fact, the conjugacy is not absolutely continuous). 

Finaly, since $\alpha$ is Diophantine (of exponent $\tau>1$), there exists a constant $C_\alpha>0$ such that, for all  $(x,y) \in \T^2$ and all $n\in \N$, 
$$
  \label{eq:Khanin_CC}
 \min \{   \text{dist}_{\T^2}  (T_1^n (x,y), (x,y) ) ,   \text{dist}_{\T^2}  (T_2^n (x,y), (x,y) )\}    \geq   \text{dist}_{\T}  (R_\alpha^n x, x) \geq    C_\alpha n^{-\tau}   \,.
$$
so that the condition on close returns in formula~\eqref{eq:Khanin_C}  is verified.  We have thus constructed a counterexample to Khanin's conjecture and our argument
is complete.
 \end{proof}
 
 \section{Time changes and Skew-product flows}
 \label{sec:time_skew}
 
 In this section we introduce time-changes and skew product flows and examine the relation between conjugacies and cohomology, focusing in particular on 
 time changes of horocycle flows and  nilflows and on skew product flows with circle fibers over horocycle flows, nilflows and Kochergin flows.
 
 \subsection{Time changes} 
 
Let $\phi_\R$ denote a smooth flow on a compact manifold $M$  and  let $f \in C^0(M)$ be a strictly positive function. Let $\phi^f_\R$ be the corresponding 
time change of a $\phi_\R$ on $M$ . We recall that the flow $\phi^f_\R$ is defined as follows:
$$
 \phi^f_t(x) :=\phi_{w(x,t)}(x),  \quad \text{ for all } (x,t) \in M \times \R\,,
$$
with $w: M\times \R \to \R$ the unique function satisfying the identity
$$
  \int_0^{w(x,t)} f\left( \phi_u (x) \right) \, d u =t \,, \quad \text{ for all }  (x,t)\in M \times \R\,.
$$
We review well-known results about conjugacy of time changes and cohomology.

 \begin{lemma} 
  \label{lemma:time_changes1} Let $\phi_\R$ be any smooth flow and let $f, g \in C^k (M)$ be positive functions.  
If the cohomological equation
 $$
 u \circ \phi_t - u =  \int_0^t  (f-g) \circ \phi_s ds \,, \quad \text{ for all } t \in \R\,,
 $$
  has a solution $u \in C^l(M)$,  then the time changes $\phi^f_\R$ and  $ \phi^g_\R$ are $C^l$ conjugate. 
 \end{lemma}
 We refer to \cite{AFRU21}, Lemma 2.1,  for a proof of this classical fundamental result in the measurable case.  It follows immediately
 from the argument that the conjugacy is as regular as the solution of the cohomological equation. 

 \smallskip 
Whether any conjugacy of time changes is necessarily given by cohomology is  a deeper question, especially for measurable conjugacies. 

A classical result of M.~Ratner \cite{Rat87}, revisited in \cite{FlaFo19}  (see also \cite{KLU20} for a completely different approach) answers the above question in the affirmative for
time-changes of a horocycle flow.

We recall that a {\it classical horocycle flow}  on the unit tangent bundle of a hyperbolic surface can be smoothly identified  with the homogeneous flow generated by a unipotent vector field in the Lie algebra of the Lie group $PSL(2, \R)$ on the homogeneous space $\Gamma \backslash PSL(2, \R)$, quotient of the group over a lattice $\Gamma <PSL(2,\R)$.

\begin{theorem} (\cite{Rat82}, Th. 2 and Th.3)  \label{thm:Ratner_rig}  Let $\phi_\R$ be the classical horocycle flow on a compact quotient $M:= \Gamma \backslash PSL(2, \R)$ 
and let $f$ and $g$ be positive integrable smooth functions with the same mean on $M$.  If the time changes $\phi^f_\R$ and $\phi^g_\R$ are measurably conjugate, then there exist a map $\psi_C: M\to M$, projection of the inner automorphism defined by a group element $C \in PSL(2, \R)$  with $C\Gamma C^{-1} =\Gamma$, and a measurable 
function $u:M\to \R$ such that
$$
u \circ \phi_t  - u  = \int_0^t  (f \circ \psi_C  - g) \circ \phi_s ds \,.
$$
In addition, if  for a given $T>0$ the time $T$-maps $\phi^f_T$ and $\phi^g_T$ of the time changes $\phi^f_\R$ and $\phi^g_\R$  are isomorphic, then the time changes  are isomorphic as flows and the above conclusion still holds.

\end{theorem} 

 A result analogous to Ratner's above theorem can be derived for general nilflows as a corollary of the mixing result for time changes proved in
 \cite{AFRU21}.  
 
 We recall the a {\it nilflow} is a homogeneous flow generated by an element of a Lie algebra of a nilpotent Lie group $N$ on the homogenous
 space $M=\Gamma \backslash N$,  quotient of the group over a lattice $\Gamma <N$, which is called a {\it nilmanifold}. All finite volume nilmanifolds are compact.
 
 \begin{theorem}  \cite{AFRU21} For every uniquely ergodic nilflow $\phi_\R$ on any compact nilmanifold $M$ there exists a dense subspace $\mathcal P \subset C^\infty(M)$ 
 (trigonometric polynomials) such that for any positive $f\in \mathcal P$ either $f$ is measurably cohomologous to a constant or the  time change $\phi^f_\R$ is mixing.
 In particular, the time change $\phi^f_\R$  is measurably conjugated to the nilflow $\phi_\R$  if and only if $f$ is measurably cohomologous to a constant.
  \end{theorem} 
 
 For general smooth time changes and only in the Heisenberg case we can prove the following weak mixing result:
 
  \begin{theorem} 
   For almost every Heisenberg  nilflow $\phi_\R$ on a Heisenberg nilmanifold $M$ and for any positive function $f\in C^\infty(M)$  either 
   the time change $\phi^f_\R$ is weakly mixing or   $f$ is smoothly cohomologous to constant.  In particular, 
 time change $\phi^f_\R$  is measurably conjugated to the nilflow $\phi_\R$  if and only if $f$ is smoothly cohomologous to constant.
  \end{theorem} 
 \begin{proof}
 We assume that $\phi^f_\R$ is not weakly mixing and derive that $f$ is smoothly cohomologous to a constant. Let $u\in L^2(M)$ denote a square
 integrable eigenfunction for $\phi^f_\R$ of eigenvalue $\imath \lambda \not =0$.  Since $\phi^f_\R$ is by definition generated by the vector field
 $f^{-1} X$, the eigenvalue equation can be written as $Xu = \imath \lambda f u$ in terms of the generator $X$ of the Heisenberg nilflow. 
 After integration,  for all $t\in \R$ we have the identity
 $$
u \circ \phi_t  =    u \exp ( \imath \lambda \int_0^t  f\circ \phi_s ds ) 
 $$
 Let $\phi^Z_\R$ denote the flow of the central vector field (which commutes with $\phi_\R$). We also have, for all $z \in \R$, 
 $$
 u \circ \phi^Z_z \circ  \phi_t  -  u \circ \phi_t   =    u \left[ \exp ( \imath \lambda \int_0^t  f \circ \phi^Z _z \circ \phi_s ds )   -  \exp ( \imath \lambda \int_0^t  f  \circ \phi_s ds )    \right]
 $$
By continuity in mean, since $u\in L^2(M)$ we have, uniformly with respect to $t\in \R$, 
 $$
 \lim_{z \to 0}  \Vert u \circ \phi^Z_z \circ  \phi_t  -  u \circ \phi_t  \Vert_{L^2(M)}   = \lim_{z \to 0}  \Vert u \circ \phi^Z_z   -  u   \Vert_{L^2(M)}   =0\,,
 $$
hence the following equation holds in measure:
\begin{equation}
\label{eq:meas_lim} 
 \lim_{z \to 0} \lim_{t\to +\infty}  \exp \left [ \imath \lambda (\int_0^t  f \circ \phi^Z _z\circ \phi_s ds   -  \int_0^t  f  \circ \phi_s ds) \right] \,= \, 1.
\end{equation}
By the intermediate value theorem we have, for all $t>0$, $x\in M$ and $z \in [-1, 1]$, 
\begin{equation}
\label{eq:int_value}
\begin{aligned}
z &\left\vert  \int_0^t  Zf  \circ \phi_s(x) ds \right\vert  -\frac{z^2}{2}  \Vert \int_0^t  Z^2 f \circ\phi_s ds \Vert_{C^0(M)}  \\   &  \qquad \leq  \left\vert \int_0^t  f \circ \phi^Z _z \circ \phi_s(x) ds -  \int_0^t  f \circ \phi_s (x) ds \right\vert  \leq    z  \Vert \int_0^t  Z f \circ\phi_s ds \Vert_{C^0(M)}  \,.
\end{aligned}
\end{equation}
  By the results on the cohomological equation of Heisenberg nilflows, under a Diophantine condition on the nilflow,  the function $f$ is cohomologous to a 
 constant if and only if the function $Zf$ is a coboundary, hence we can assume that $Zf$ is not a coboundary.  Under this assumption, there exists an
 $X$-invariant distribution $D \in W^{-r} (M)$ (for all $r>1/2$) such that $D(Zf) \not =0$, and there exists a constant $C>0$ such that,  on the one  hand, for all $t>1$,
 $$
\Vert   \int_0^t  Z  f  \circ \phi_s ds \Vert_{L^2(M)}  \geq  C \vert D(Zf) \vert  \,  t^{1/2} \,.
 $$
 on the other hand, there exists a diverging sequence $\{t_n\}$ (which depends only on the returns of the orbit of the nilflow under renormalization to a compact set
 in the moduli space) and, for $r > 5/2$, there exists a constant $C_{r}>0$ such that 
 \begin{equation}
 \label{eq:unif_bounds} 
\max_{k\in \{1,2\}}  \Vert   \int_0^{t_n}   Z^k  f  \circ \phi_s ds  \Vert_{C^0(M)}  \leq  C_r \Vert f \Vert_{W^r(M)}    \,  t_n^{1/2}  \,.
 \end{equation}
 It follows that, there exists a constant $C'>0$ and, for each $n\in \N$, there exists a set $E_n\subset M$ such that  $\text{vol} (E_n) \geq  C' \vert D(Zf) \vert^2 / \Vert f \Vert^2_{W^r(M)} $ and, in addition,
 $$
  \vert  \int_0^{t_n}  (Z  f  \circ \phi_s)(x)  ds \vert   \geq  C' \vert D(Zf) \vert  \, t_n^{1/2} \,,  \quad \text{ for all } x \in E_n\,.
 $$
 Finally,  by the above intermediate value formula \eqref{eq:int_value} and by the upper bounds \eqref{eq:unif_bounds} (in the uniform norm)  on ergodic integrals, it follows that 
 there exist constants  $c_f$, $c'_f >0$ such that, for the for the sequence $z_n= c_f  \lambda^{-1}  t_n^{-1/2}$ we have, for all $n\in \N$, 
 $$
  \frac{c'_f}{\lambda}   \leq    \left\vert \int_0^{t_n}  (f \circ \phi^Z _{z_n} \circ \phi_s)(x) ds -  \int_0^{t_n}   (f \circ \phi_s)(x) ds \right \vert  \leq  \frac{1}{2\lambda} \,, \quad \text{ for all } x \in E_n\,,
 $$
 in contradiction with the limit in measure of formula \eqref{eq:meas_lim}. The argument is therefore complete.
  \end{proof}

 \subsection{Skew product flows}  
 
 We prove basic results about conjugacy of skew product flows and cohomology which we could not find in the literature.

 Let $\phi_\R$ denote a smooth flow on a compact manifold $M$ and let $f \in C^0(M)$ be a  strictly positive function. 
 
 We define the skew-product flow $\Phi^f_\R$ on $M\times \T$  as follows
 $$
 \Phi^f_t (x, \theta)  =  \left(\phi_t (x) ,  \theta + \int_0^t  (f\circ \phi_s)(x) ds \right)\,,   \quad \text {for all }  (x, \theta, t) \in M\times \T \times \R\,.
 $$

 \begin{lemma} 
  \label{lemma:conj_skew_shifts1} Let $\phi_\R$ be any flow on $M$ and let $f, g \in C^k (M)$ be positive functions.  
If the cohomological equation
 $$
 u \circ \phi_t - u =  \int_0^t  (f-g) \circ \phi_s ds \,, \quad \text{ for all } t \in \R\,,
 $$
  has a solution $u \in C^l(M)$,  then the  flows $\Phi^f_\R$ and  $ \Phi^g_\R$ are $C^l$ conjugate. 
 \end{lemma}
 \begin{proof}
  Let us assume that the cohomological equation as a solution $u\in C^l(M)$. We prove that the flows $\Phi^f_\R$ and  $ \Phi^g_\R$ 
are conjugate by a conjugacy of the regularity of the solution $u$ of the cohomological equation. Let us define the homeomorphism
$h : M\times \T \to  M\times \T$ by the formula
$$
h(x, \theta) =  (x, \theta + u(x) ) \,, \quad \text{ for all } (x,\theta) \in M\times \T\,.
$$
We have, for all $(x, \theta) \in M\times \T$, 
$$
\begin{aligned} 
(\Phi^f_t \circ h) (x, \theta) &= \left( \phi_t(x), \theta + u(x) +  \int_0^t  (f\circ \phi_s)(x) ds ) \right)  \\ &= \left( \phi_t(x), \theta + u \circ \phi_t(x) + 
 \int_0^t  (g\circ \phi_s)(x) ds \right)  = (h \circ \Phi^g_t)(x, \theta)\,,
\end{aligned} 
$$
which proves the statement. 
 \end{proof}
 
 Let $X$ denote the generator of the smooth flow $\phi_\R$ on $M$ and let $\mathcal Z_X$ denote the centralizer of $X$ in the Lie algebra
 $\mathcal V (M)$ of smooth vector field on $M$:
 $$
\mathcal Z_X:=\{  Y \in \mathcal V (M) \vert  [X,Y] =0\}\,.
 $$
 
 \begin{definition}
 \label{def:expansive}  The smooth flow $\phi_\R$  with generator $X$ on $M$ is called expansive if the following holds: there exists 
 $\delta >0$ such that for all $x$, $x' \in M$ such that $x' \not \in  \exp (\mathcal Z_X) x$  (that is, $x'$ does not belong to the orbit of 
 $x$ under the centralizer $\mathcal Z_X$) there exists  $t \in \R\setminus \{0\}$  such that (with respect to a given smooth distance function on $M$) we have
 $$
 \text{dist}_M (\phi_t(x), \phi_t(x') )  \geq  \delta\,.
 $$
 \end{definition}

 \begin{lemma} 
 \label{lemma:conj_skew_shifts2}
  Let the smooth flow $\phi_\R$ be expansive, and let $f, g \in C^k (M)$ be positive functions.  If the flows $\Phi^f_\R$ and  $ \Phi^g_\R$ are $C^l$ conjugate  (with $l\leq k$) and
minimal (or quasi-minimal in the case of Kochergin flows), then there exists a  $C^l$ automorphism $\bar h: M\to M$ of the flow $\phi_\R$ (in the sense that $\bar h \circ \phi_t= \phi_t \circ \bar h$ for all $t\in \R$) 
such that the  cohomological equation
 $$
 u \circ \phi_t - u =  \int_0^t  (f \circ \bar h -g) \circ \phi_s ds \,, \quad \text{ for all } t \in \R\,, 
 $$
  has a solution $u \in C^l(M)$.
 \end{lemma} 
\begin{proof} Let us assume that there exists a conjugacy $h : M\times \T \to M\times \T$ between $\Phi^f_\R$ and $\Phi^g_\R$, that is, a homeomorphism  
$h = (h_M, h_{\T})  \in C^l(M\times \T)$ such that
$$
(\Phi^f_t \circ h) (x,\theta) = (h \circ \Phi^g_t) (x,\theta)\,, \quad \text{ for all } (x, \theta, t) \in M\times \T \times \R\,,
$$
which implies, for all $(x,\theta, t) \in M\times \T \times \R$, 
$$
\begin{aligned}
(\phi_t \circ h_M)(x, \theta)&=  h_M \left (\phi_t (x), \theta + \int_0^t (g \circ \phi_s) (x) ds  \right) \,; \\
h_\T (x, \theta) + \int_0^t (f\circ \phi_s \circ h_M)(x, \theta) ds &= h_\T \left(\phi_t(x), \theta + \int_0^t (g\circ \phi_s)(x) ds    \right) \,.
\end{aligned}
$$
We claim that since the flow $\phi_\R$ is expansive, for all $\theta \not = \theta' \in \T$ and for all $x\in M$, the points $h_M(x, \theta)$ and
$h_M(x, \theta')$ belong to the same orbit of the centralizer $\mathcal Z_X =\{X, Z_1, \dots, Z_l\} \subset C^1(M, TM)$  of the generator $X$ of the 
flow.

Let $\delta>0$ denote the expansiveness constant of $\phi_\R$. Since the space $M\times \T$ is compact and  the map $h_M: M\times \T \to M$ is continuous, 
it is uniformly continuous. By definition of uniform continuity, with respect to any given distance function $\text{dist}_M $
on $M$ and a translation invariant distance function $ \text{dist}_\T$ on $\T$, there exists $\eta>0$ such that, for all $(y, \theta)$, $(y', \theta') \in M\times \T$ we have
$$
\text{dist}_M (y, y') + \text{dist}_\T (\theta, \theta') <\eta \implies   \text{dist}_M \left (h_M(y, \theta), h_M(y', \theta') \right)  <\delta\,.
$$
By subdividing the arc $[\theta, \theta']$ into finitely many sub-arcs, we can assume that $\text{dist}_\T (\theta, \theta') < \eta$ so that
$$
\begin{aligned}
\text{dist}_M (\phi_t(x), \phi_t(x)) &+ \text{dist}_\T \left(\theta + \int_0^t (g\circ \phi_s)(x) ds , \theta' + \int_0^t (g\circ \phi_s)(x) ds\right)  \\ &= 
\text{dist}_\T (\theta, \theta' ) < \eta
\end{aligned} 
$$
which implies that, for all $t \in \R$, 
$$
 \text{dist}_M \left (h_M\Big(\phi_t(x), \theta + \int_0^t (g\circ \phi_s)(x) ds  \Big), h_M\Big(\phi_t(x), \theta' + \int_0^t (g\circ \phi_s)(x) ds \Big) \right)  <\delta \,.
$$
Finally, by the first of the above conjugation identities we derive that, for all $t\in \R$,
$$
 \text{dist}_M  \left ((\phi_t \circ h_M)(x, \theta),  (\phi_t \circ h_M)(x, \theta') \right)  < \delta
$$
hence, by the expansiveness property of $\phi_\R$, the points  $h_M(x, \theta)$ and $h_M(x, \theta')$ belong to the same orbit of the
centralizer of its generator. 

As a consequence of the claim proved above, there exist continuous functions $r, s_1, \dots s_l: M\times \T \to  \R$ such that 
$$
h_M(x, \theta) =   \exp (  \sum_{i=1}^l  s_i(x, \theta) Z_i) \circ \phi_{ r(x, \theta)}  (h_M(x,0)) \,, \quad \text{ for all }  (x, \theta) \in M\times \R \,.
$$
Let us adopt the notations ${\bf s} = (s_1, \dots, s_l) \in \R^l$ and ${\bf Z} =(Z_1, \dots, Z_l) \in C^1(M, TM)^l$ so that we can write
$$
{\bf s} \cdot {\bf Z} = \sum_{i=1} ^l   s_i Z_i\,.
$$
From the first conjugation identity we derive
$$
(  \exp ( {\bf s} (x, \theta) \cdot {\bf Z} ) \circ \phi_{t+ r(x, \theta)}  \circ h_M)(x, 0)=     \exp ( {\bf s} \circ \Phi^g_t(x, \theta) \cdot {\bf  Z}) \circ \phi_{  r \circ \Phi^g_t (x, \theta) } \circ h_M  (\phi_t (x), 0) \,,
$$
which in turn implies, for all but finitely many $x\in M$, that  the functions
$$
r \circ \Phi^g_t (x, \theta) -  r(x, \theta) \quad \text{and} \quad   {\bf s} \circ \Phi^g_t (x, \theta) -  {\bf s}(x, \theta) 
$$
do not depend on $\theta \in \T$.  In fact, from the above identity we immediately derive
$$
(  \exp ( ({\bf s} - {\bf s} \circ \Phi^g_t) (x, \theta) \cdot {\bf Z} ) \circ \phi_{t+ (r-r \circ \Phi^g_t)(x, \theta) }  \circ h_M)(x, 0)=   h_M  (\phi_t (x), 0) \,,
$$
and by the (quasi)-minimality of the flow $\phi_\R$ the map
$$
({\bf s}, t)   \to   (  \exp ( {\bf s}  \cdot {\bf Z} ) \circ \phi_{t  } ) (x_0)    \in M
$$
is an immersion for all but finitely many point $x_0 \in M$. 

\smallskip
By the definition of $\Phi^g_\R$, it follows that  by defining
 $$
 \rho(x) =  \frac{1}{2\pi}\int_0^{2\pi} r (x, \theta) d\theta \quad\text{and} \quad  \sigma(x) =  \frac{1}{2\pi}\int_0^{2\pi} {\bf s} (x, \theta) d\theta \,, \quad \text{ for all } x\in M\,,
 $$
 we derive the identities
$$
\begin{aligned}
r \circ \Phi^g_t (x, \theta) -  r(x, \theta) = \frac{1}{2\pi}\int_0^{2\pi} [r \circ \Phi^g_t (x, \theta) -  r(x, \theta) ]d\theta =    \rho \circ \phi_t(x) - \rho (x) \,, \\
{\bf s} \circ \Phi^g_t (x, \theta) -  {\bf s} (x, \theta) = \frac{1}{2\pi}\int_0^{2\pi} [{\bf s} \circ \Phi^g_t (x, \theta) -  {\bf s}(x, \theta) ]d\theta =    \sigma \circ \phi_t(x) - \sigma (x) \,.
\end{aligned}
$$
From the first conjugation identity, we have thus derived that
$$
\begin{aligned}
\phi_t  [ \exp ( \sigma(x) \cdot {\bf Z}  ) & \phi_{\rho(x)}  (h_M(x, 0) ) ] \\  &=  \exp ( \sigma(\phi_t(x)) \cdot {\bf Z}  )  \phi_{\rho (\phi_t(x)) } [ h_M  (\phi_t (x), 0)] \,,
\end{aligned} 
$$
that is the map $\bar h : M\to M$ defined as $\bar h(x) =  \exp ( \sigma(x) \cdot {\bf Z}  )  \phi_{\rho(x)} [ h_M(x, 0)]$ is an automorphism of the flow $\phi_\R$, and we can write
$$
h_M(x, 0) =  \exp ( -\sigma(x) \cdot {\bf Z}  )  \phi_{-\rho(x)}   \circ \bar h ( x)  \,, \quad \text{ for all } x \in M\,.
$$
The first conjugation identity finally implies that, for all $(x,t) \in M\times \R$ we have 
$$
\begin{aligned}
\exp \left [({\bf s} (x, \theta) -\sigma(x) ) \cdot {\bf Z} \right ]& \circ  \phi_{r(x, \theta) -\rho(x) } ( \bar h ( x))  \\ =    &\exp \left [ ({\bf s} \circ \Phi^g_t(x, \theta) -\sigma(x) ) \cdot {\bf Z} \right] \circ \phi_{ r \circ \Phi^g_t (x, \theta) - \rho (\phi_t(x))  }( \bar h ( x)) \,,  
\end{aligned}
$$
which by the minimality of $ \Phi^g_\R$ implies that  the functions $r -\rho$ and ${\bf s} -\sigma$  are constant on $M\times \T$. By the definition of
$\rho$ and $\sigma$ as the average of $r$ and ${\bf s}$, respectively, along the fibers of the fibration $M\times \T \to M$, the functions $r-\rho$ and ${\bf s} -\sigma$ have zero average,  hence
$r =\rho$ and ${\bf s}= \sigma$ are functions on $M$ (independent of $\theta \in \T$).  We conclude that, for all $(x, \theta)\in M\times \T$, we have
$$
\begin{aligned}
h_M(x, \theta) &= \exp( {\bf s}(x, \theta) \cdot {\bf Z}  )\circ  \phi_{ r(x, \theta)}  (h_M(x,0))   \\ &=  \exp( ({\bf s}(x, \theta) -\sigma(x)) \cdot {\bf Z}  )\circ  \phi_{ r(x, \theta) -\rho(x)} ( \bar h(x) ) = \bar h (x)  \,.
\end{aligned} 
$$
Since $h: M\times \T \to M\times \T$ is a homeomorphism, there exists  $U: M\times \T \to \R$  of class $C^l$ such that
$$
h_\T (x, \theta) =   \theta +  U (x, \theta)  \,\, (\text{mod. } \Z) \,, \quad \text{ for all }  (x, \theta) \in M\times \T\,.
$$
By the second conjugation identity we have  (mod. $\Z$) 
$$
U (x, \theta) + \int_0^t (f\circ \bar h \circ  \phi_s)(x) ds = \int_0^t (g\circ \phi_s)(x) ds+   U  \left(\phi_t(x), \theta + \int_0^t (g\circ \phi_s)(x) ds    \right) 
$$
Let then  $u\in C^l(M)$ be the real valued function defined as 
$$
u (x) := \frac{1}{2\pi}  \int_0^{2\pi}   U(x, \theta)  d\theta\,, \quad \text{ for all }  x\in M\,.
$$
By integrating the above  conjugation identity and by continuity, we derive that there exists an integer $m\in \Z$ such that, for all  $(x, t) \in M\times \R$,
$$
u (\phi_t(x)) - u(x) = m+  \int_0^t (f\circ \bar h \circ \phi_s)(x) ds - \int_0^t (g\circ \phi_s)(x) ds \,.
$$
however clearly for $t=0$ the above identity yields $m=0$, hence the cohomological equation has the solution $u\in C^l(M)$,
as stated.

\end{proof} 

The above result motivates the following discussion concerning the centralizer, that is, the automorphism group, of parabolic flows. For horocycle flows, the now 
classical work of Ratner \cite{Rat82}, \cite{Rat87} completely describes the measurable essential centralizer, in terms of the normalizer in $PSL(2, \R)$ of the fundamental 
group of the hyperbolic surface. In particular, the essential centralizer is finite.  For general flows with the Ratner property, it was proved in~\cite{KL17} that the essential 
centralizer is discrete at most countable. It is not known whether it can ever be infinite in this case. 

\medskip
We define the  {\it Kochergin flows} on orientable higher genus surfaces, as the class of locally Hamiltonian (area-preserving) smooth flows with (possibly degenerate) 
saddle singularities which are smooth time changes of  a translation flow, with time change function vanishing at the appropriate order at the singular points of the translation
flow.  For uniquely ergodic Kochergin flows on orientable higher genus surfaces we have the following result. 

\begin{lemma} 
\label{lemma:Koch_centr} Any mapping class which preserves a uniquely ergodic measured foliation is periodic. In particular, 
the $C^0$ essential centralizer of a uniquely ergodic Kochergin flow on an orientable higher genus surface has a periodic mapping class,
therefore for the generic Kochergin flow the $C^0$ essential centralizer is trivial. 
\end{lemma} 
\begin{proof} Let $h:M \to M$  be a homeomorphism which preserves a uniquely ergodic measured foliation,  including
its transverse invariant measure. By Nielsen-Thurston's classification, $h$ must be isotopic to a finite order element of the mapping class group. In fact, it cannot be isotopic to a pseudo-Anosov homeomorphism since it must have no dilation, and it cannot be non-periodic and reducible since the orbit foliation is ergodic. 
Indeed, if $h$ leaves invariant a finite set of closed curves, then
an iterate of $h$ has a fixed loop $\Gamma$ on which the action of $h$ is a homeomorphism of the circle, which preserves a non-atomic invariant measures on the complement of
closed. Since singular leaves are preserved, the restriction of $h$ to the loop $\Gamma$ has rational rotation number.  It follows that an iterate of $h$ is the identity on a closed 
subset of $\Gamma$, hence it is the identity on every corresponding leaf of the orbit foliation, hence on $M$ since the foliation is quasi-minimal.

A periodic and non-reducible mapping class leaves invariant a conformal structure. There is a unique measured foliation $\mathcal F^\perp$, transverse to the orbit foliation $\mathcal F$ of the Kochergin flow such that $(\mathcal F, \mathcal F^\perp)$ is a translation structure with that conformal class.  Such a translation structure is therefore fixed
by the periodic mapping class, and it is therefore a branched cover. The set of such translation structures has measure zero in the moduli space of Abelian differentials.

Finally, let $h$ a homemorphism which preserves the orbit foliation, which mapping class is the identity, that is, $h$ is isotopic to the identity. We can assume, up to considering powers,  that $h$ fixes all singular points and singular leaves,  hence it fixes the foliation. Since it commutes with the flow, by quasi-minimality it is a time $t$-map of the flow itself. 
\end{proof}

It seems harder to describe the essential measurable centralizer of  uniquely ergodic Heisenberg nilflows.  It was proved in \cite{FoKa20b} that non-trivial times changes of Heisenberg nilflows of {\it bounded type} (such that the projected linear flow in $\T^2$ has a rotation vector with slope of bounded type) satisfy the Ratner property, hence by \cite{KL17} they have discrete at most countable essential centralizer.  However, for the continuous centralizer we have 

\begin{theorem}  (\cite{KMS91}, Cor. 4.5)
\label{thm:Heis_centr}
Let $\phi_\R$ be a uniquely ergodic (or Diophantine) Heisenberg nilflow on a Heisenberg manifold $M:= \Gamma\ H_3(\R)$, quotient of the
$3$-dimensional Heisenberg group $H_3(\R)$ by a (co-compact) lattice $\Gamma < H_3(\R)$. For every element $h:M\to M$ of the $C^0$
essential centralizer of $\phi_\R$ there exists an element $g\in H_3(\R)$ in the normalizer of $\Gamma$, that is, such that $g\Gamma g^{-1} =\Gamma$, with 
$$
h (\Gamma x) =  \Gamma g x \,,  \quad \text{ for all }   x \in H_3(\R)\,.
$$
We note that in contrast with the case of the horocycle flow, the normalizer of $\Gamma$ contains a one-parameter subgroup: the center of the Heisenberg group.
\end{theorem} 

In analogy with Ratner's result for horocycle flows we propose the following:
\begin{conjecture} 
\label{conj:Heis_centr}
Let $\phi_\R$ be a uniquely ergodic (or Diophantine) Heisenberg nilflow on a Heisenberg manifold $M:= \Gamma\slash H_3(\R)$, quotient of the
$3$-dimensional Heisenberg group $H_3(\R)$ by a (co-compact) lattice $\Gamma < H_3(\R)$.  The essential measurable centralizer coincides
with essential $C^0$ centralizer.
\end{conjecture} 

 \section{Non-smooth continuous solutions of cohomological equations}
 \label{sec:cohom}
 
 In this section we recall known results about the existence of continuous, non-smooth solutions of cohomological equations for horocycle flow
 and locally Hamiltonian flows on surfaces. In general, for a smooth flow $\phi_\R$ with generator $X$ on a manifold $M$, for any measurable function $f$ on $\R$ 
 the (scalar) cohomological equation
 $$
 u\circ \phi_t -u =  \int_0^t  f\circ \phi_s ds \,,   \quad \text{ for all } t\in \R\,,
 $$
is equivalent to the linear partial differential equation (taken in weak sense)
 $$
 Xu=f \,.
 $$
 Following A. Katok (see \cite{Kat03}, section 3), the flow $\phi_\R$ is called {\it stable} on the function space $W$ if the range of the Lie derivative operator
 $\mathcal L_X$ in $W$ is a closed subspace of $W$, and it is called $W'$-stable on $W$ if every element $f\in W$ which belongs to the range 
 of $\mathcal L_X$ is the image of $u \in W'$.  Whenever $W'$ and $W$ are spaces of smooth functions, and $C^\infty(M) \subset W' \subset W$,
 by the Hahn-Banach theorem, under the hypothesis that $\phi_\R$ is $W'$-stable on $W$, the above cohomological equation  has a solution
 $u \in W'$ for all $f\in W$ in the kernel of all invariant distributions for the flow $\phi_\R$ which belong to the dual space $W^\ast$.
 The  {\it invariant distributions} for the flow $\phi_\R$ are defined as the distributions $D \in \mathcal D' (M)$ (in the sense of L. Schwartz) such that
 $$
 X D =0 \quad  \Longleftrightarrow \quad   D(X u)=0 \,, \quad \text{for all } u \in C^\infty(M)\,.
 $$
 Almost all Kochergin flows \cite{Fo97}, \cite{MMY05}, \cite{Fo07},  classical horocycle flows \cite{FlaFo03} and Heisenberg nilflows \cite{FlaFo06}  are stable 
 on Sobolev spaces of sufficiently high order.

 \subsection{Horocycle flows}  
 
Let $\phi^U_\R$ be the horocycle flow on the unit tangent bundle $M=T_1(S)$ of a compact hyperbolic surface. It is well-known that the horocycle
flow is a unipotent flow on a homogeneous space $\Gamma \backslash PSL(2, \R)$ for a (co-compact) lattice $\Gamma < PSL(2, \R)$.
Let $U$ denote the generator of the horocycle flow.  For all $k\in \N$ let $W^k(M)$ be the $L^2$ Sobolev space of all functions with  square-integrable
derivatives up to order $k$ on the compact manifold $M$.

The rigidity theorem by Ratner was complemented in \cite{FlaFo19} with the following regularity lemma:
\begin{lemma} 
\label{lemma:cocycle_rigidity}
  If a function $f \in W^r(M)$,  with $r>2$, is a
  coboundary for the horocycle flow $\phi_\R$ with measurable primitive $u: M \to \R$ 
then the  primitive $u$ is continuous and belongs the Sobolev space $\bigcap_{t<1} W^t (M) $.
\end{lemma}
\begin{proof} The proof of an analogous result, which does not include the continuity of the primitive, was given in 
\cite{FlaFo19}, Lemma 2.5. We reproduce the argument for the convenience of the reader.

Let us assume that $f$ is a coboundary with measurable primitive and
  derive that in fact the primitive belongs to the space
  $\bigcap_{t<1} W^t (M)$.  By Luzin's theorem, for
  any $\epsilon >0$, there exists a constant $C_\epsilon>0$ such that
  the following holds. For any $t>0$ there exists a measurable set
  $\mathcal B_{\epsilon, t} \subset M$ of volume
  $\text{vol} (B_{\epsilon, t}) \geq 1-\epsilon$ such that
  \begin{equation}
    \label{eq:meas_cob}
    \vert \int_0^t  f (\phi_s (x) ) ds  \vert \leq C_\epsilon  \,, \quad \text{ for all } \, x \in B_{\epsilon, t}\,.
  \end{equation}
If the function $f$ is not in the kernel of all horocycle-invariant
  distributions supported on irreducible representations of the
  principal and complementary series, it follows from the results of
  the authors~\cite{FlaFo03}, Th. 5, that the $L^2$ norm of ergodic
  integrals $f$ diverge (polynomially) and every weak limit of the
  random variables
$$
\frac{ \int_0^t f (h_u x) du} { \Vert \int_0^t f \left(\phi_s (x) \right) ds
  \Vert_{L^2(M)} } \,,
$$
in the sense of probability distributions, is a (compactly supported)
distribution on the real line {\it not supported at the origin}. The
argument is given in detail in the proof of \cite{FlaFo03}, Cor. 5.6,
which established that the Central Limit Theorem does not hold for
horocycle flows. Refined theorems on limit distributions for horocycle
flows were proved in \cite{BuFo14} but under the slightly stronger
(technical) assumption that $f\in W^r(M)$ with
$r>11/2$.  It follows that in this case property \eqref{eq:meas_cob}
implies that $f$ is in the kernel of all invariant distributions
supported on irreducible representations of the principal and
complementary series.

It follows then from the results of \cite{FlaFo03} (in particular, from 
Theorem 1.1 and 1.2 and from
the formulas for invariant distributions of sections 3.1 and 3.2) that $f$ is
cohomologous to a function given by a harmonic form on the unit
tangent bundle, with a primitive in the space
$\bigcap_{ t<1} W^t(M)$.  It then  follows
from results of D.~Dolgopyat and O.~Sarig \cite{DS17} on the so-called
{\it windings} of the horocycle flow, for instance, from
\cite{DS17}, Theorems 3.2 and 5.1, or Lemma 5.10, that a non-zero harmonic form 
cannot be a co-boundary with measurable primitive. 

We have thus reduced our argument to the case of a function $f \in W^r(M)$ supported
on irreducible representations of the discrete series $H_n$ of parameter $n \geq 2$.
 By decomposition in \cite{FlaFo03}, formula (67), 
 for averages along horocycle orbits, by \cite{FlaFo03}, Lemma 5.4 and Lemma 5.7, since by assumption $\mathcal D(f) =0$ for all horocycle
 invariant distributions of Sobolev order $S_{\mathcal D}\leq 1$,  it follows that there exists a constant $C(M) >0$ such that 
 $$
 \vert  \int_0^T  f \circ \phi ^U _t(x) dt  \vert  \leq C(M) \,, \quad \text{ for all } (x,T) \in M\times \R\,.
 $$
 Since the horocycle flow is minimal (uniquely ergodic), it follows by the Gottschalk--Hedlund theorem that the cohomological equation
 has a solution $u\in C^0(M)$, as claimed.  In addition, by \cite{FlaFo03}, Theorem 1.2,  since $\mathcal D (f)=0$ for all horocycle invariant 
 distributions of Sobolev order $S_{\mathcal D} <2$,  and since by ergodicity the zero-average solution is unique,  the solution $u\in W^t(M)$ for all $t<1$.

\end{proof}

 \begin{lemma}
 \label{lemma:CE_Horo}
  There exist (positive) functions $f \in C^\infty(M)$ (in fact, real analytic) such that the cohomological equation
 $$
 u\circ \phi^U_t - u = \int_0^t  f \circ \phi^U_s ds   \,, \quad \text{ for all } t \in \R\,,
 $$
 has a solution $u\in C^0(M) \setminus W^1(M)$.  In addition, for every $k \in \N$ there exists a function $f_k  \in C^\infty (M)$ 
 (real-analytic) such that the cohomological equation has a solution $u_k \in C^0(M) \cap W^t(M) \setminus W^{k+1}(M)$, for all $t<k+1$. 
 \end{lemma} 
 \begin{proof}  The argument is based on the results of~\cite{FlaFo03} on solutions of the cohomological equation for horocycle flows.
 We follow the notations of \cite{FlaFo03}. 
 
 Let $f_n$ be a function supported on the irreducible components of the discrete series of 
 Casimir parameter  $\mu= - n^2 +n$ ($n \in \N\setminus \{0\}$) such that there exists a horocycle invariant distribution  $\mathcal D_n \in \mathcal I_n$ 
 such that $\mathcal D_n(f) \not=0$.  
 
 By Theorem 1.1 (4) of \cite{FlaFo03} all horocycle invariant distributions $\mathcal D \in \mathcal I_n$ have Sobolev order $S_{\mathcal D} =n$,
 hence by the converse Theorem 1.3 (or Lemma 4.9)  of \cite{FlaFo03}, since $\mathcal D_n(f)\not=0$,  the cohomological equation $U u= f$  has  no solution 
 $u \in W^{n-1}(M)$.  Equivalently, the integrated form of the cohomological equation, given in the statement of the lemma, has no solution $u\in W^{n-1}(M)$.
 
 By Lemma~\ref{lemma:cocycle_rigidity}  in all the above mentioned cases, whenever $n\geq 2$,  there exists a solution $u \in C^0(M)$. In addition, by \cite{FlaFo03}, 
 Theorem 1.2,  since $\mathcal D (f_n)=0$ for all horocycle invariant distributions of Sobolev order $S_{\mathcal D} <n$,  and since by ergodicity the zero-average 
 solution is unique,  the solution $u\in W^t(M)$ for all $t<n-1$.
\end{proof}

  \subsection{Kochergin  flows}  
  
  Let $\phi_\R$ be a (Kochergin)  locally-Hamiltonian flow on an orientable surface of higher genus with saddle-like singularities at a (finite) set $\Sigma$.   
  Let  $X$ denote the generator of the flow and $\omega$ the smooth invariant area form.
  
   Let $\eta_X :=\imath_X \omega$ the unique closed $1$-form whose
  kernel coincides with  the orbit foliation of the flow. 
  The cohomology class $[\eta_X] \in H^1(M, \Sigma), \R)$, called the Katok fundamental class,
  is a complete local invariant of  the smooth conjugacy class of the flow with the space of locally Hamiltonian vector fields which coincide on a 
  neighborhood of the singularity set $\Sigma$. 
  
  Under the hypothesis that the Kochergin flow is quasi-minimal (in the sense that all non-singular semi-orbits are dense in $M$), there exists a translation 
  structure on $M$ (a complex structure and a holomorphic $1$-form on $M$) with horizontal vector field $\bar X$ and a smooth (real-analytic) function 
  $W: M \to \R^+$ which vanishes at finite order at $\Sigma$ such that   $X= W \bar X$. Let $\bar X^\perp$ denote the orthogonal vector field such that $\{\bar X, \bar X^\perp\}$ 
  is a positively oriented frame for the translation structure.
  
  For all $r>0$ let $W^r(M)$ denote the weighted Sobolev space for the translation structure $\{\bar X, \bar X^\perp\}$, introduced in \cite{Fo97}, \cite{Fo07}.
  We recall that the space $W^1(M)$ is equivalent to the standard Sobolev space on the compact surface $M$.
  
  \begin{lemma} 
   \label{lemma:CE_Koch}
  For almost all fundamental classes $[\eta_X] \in H^1(M, \Sigma, \R)$, there exist (positive) functions $f \in C^\infty_0(M\setminus \Sigma)$
  such that the cohomological equation
 $$
 u\circ \phi^X_t - u = \int_0^t  f \circ \phi^X_s ds   \,, \quad \text{ for all } t \in \R\,,
 $$
 has a solution $u\in C^0(M) \setminus W^1(M)$.  In addition, for every $k \in \N$ there exists a function $f_k  \in C^\infty_0 (M\setminus \Sigma)$ 
 such that the cohomological equation has a solution $u_k \in C^0(M) \cap W^k (M) \setminus W^{k+1} (M)$.  In fact,  there exists 
  $\lambda \in (0,1)$ such that     $u_k  \in W^{k+ r}(M)$ for any $r \in (0,1)$ such that $r <  \lambda$. 
  \end{lemma} 
  \begin{proof} 
  For functions supported away from the singularity set $\Sigma$ (in fact, more generally for functions satisfying a finite order vanishing condition on their jet at $\Sigma$), the cohomological equation for the Kochergin  flow is equivalent to the cohomological equation for  a translation flow. Indeed,  since the cohomological equations
  $$
  X u = f  \quad \text{ and } \quad    \bar X u = W^{-1}  f 
  $$
  are equivalent, the problem is reduced to characterizing solutions of the cohomological equation $\bar X u= W^{-1}f$ for the translation flow 
  and for the function $W^{-1} f \in C^\infty_0 (M\setminus \Sigma)$ (more generally, for $W^{-1} f  \in W^r(M)$) .
  
  \smallskip
  The cohomological equation for translation flows can then be reduced to a cohomological equation for its return map, which is an Interval Exchange Transformation.
  It follows from \cite{MMY12}, \cite{MY16} and \cite{Fo07}  that the space of obstructions  to the existence of a solution $u\in C^0(M)$ has dimension equal to the genus 
  $g$ of the surface (since the Kontsevich--Zorich cocycle is non-uniformly hyperbolic on strata \cite{Fo02}).  In \cite{Fo07} it was proved only that the solution is bounded.
  However, this conclusion is derived from a uniform bound on ergodic integrals (see \cite{Fo07}, Th. 5.10) and continuity follows by applying the Gottshalk--Hedlund 
  theorem with respect to an appropriate topology. Indeed, the translation flow and its return map are not continuous with respect to the topology of the surface or the 
  induced of a transverse interval, but it is continuous and minimal with respect to an appropriate topology given by a Denjoy-type construction, as in \cite{MMY05}, 
  Section 2.1.
  
  Let then $f\in C^\infty_0(M\setminus\Sigma)$ in the kernel of the above mentioned $g$-dimensional obstruction, so that the corresponding cohomological equation has
  a continuous solution.   Let 
  $$
  \lambda_1 =1 \geq \lambda_2 \geq ... \geq \lambda_g>0 >  \lambda_{g+1} \geq \dots \geq \lambda_{2g} =-1\,,
  $$
  denote the Kontsevich--Zorich exponents for the  relevant connected component of a stratum of the moduli space of Abelian differential.
  
  The Sobolev order of all invariant distributions was computed in \cite{Fo07}, Th.~B2: there exists a basis  $\{{\mathcal D}_{i,j}\}$ of the space of invariant distributions vanishing on constant functions such that the Sobolev order ${\mathcal O}(  {\mathcal D}_{i,j} )$ of the distribution ${\mathcal D}_{i,j}$ is given by the formula
  $$
  \mathcal O(  {\mathcal D}_{i,j}) = j+1 - \lambda_i \,, \quad \text{ for all } \,\,i\in \{2, \dots, 2g-1  \}, \,\, j \in \N\cup \{0\}\,. 
  $$
  Let $\bar X^\perp$ denote the translation vector field orthogonal to $\bar X$ with respect to the translation structure. We recall that we have
  $$
  {\mathcal D}_{i,j}  =   (\bar X^\perp)^j  {\mathcal D}_{i,0} \,, \quad \text{ for all }   \,\,i\in \{2, \dots, 2g-1  \}, \,\, j \in \N\cup \{0\}\,. 
  $$
  The set $ \mathcal I^1_{\bar X}:= \{1, {\mathcal D}_{2,0} , \dots , {\mathcal D}_{g,0}\}$ is a basis of the subspace of all $\bar X$-invariant distributions of 
  Sobolev order $\leq 1$, which  form the obstructions to the existence of continuous solutions.   
  
  By the deviation results \cite{Fo02}, the ergodic integrals of  any function $F\in H^1(M)$ which does not belong to the joint kernel of the space $\mathcal I^1_{\bar X}$ 
  have uniform polynomial growth on a subset of $M$ of uniformly positive measure along a subsequence of times. In particular, it follows that the cohomological equation 
  $\bar X U= F$ does not have any solution  $U\in L^2(M)$.  
  
  Let then $f_i\in C_0^\infty(\setminus\Sigma)$ be a function in the kernel of $\mathcal I^1_{\bar X}$, so that
  as pointed above the cohomological equation $\bar X u_i=f_i$ has a continuous solution, but with the property that there exists $i \in \{2, \dots, g\}$ 
  such that
  $$
  \mathcal D_{i, 1} ( f_i ) \not =0 \quad \text{ and } \quad  \mathcal D_{a, 1} ( f_i ) =0\,,  \quad \text{for }  a\not =i  \in \{2, \dots, 2g-1\}\,.
  $$
 If follows that the unique zero-average solution $u_i\in C^0(M)$ of the equation $\bar X u_i=f_i$ does not belong to the Sobolev space
 $W^1(M)$, otherwise $U_i:=\bar X^\perp u_i \in L^2(M)$ would be a solution of the cohomological equation $\bar X U_i = X^\perp f_i$, which has 
 no square-integrable solutions since
 $$
  \mathcal D_{i, 0} ( X^\perp f) =  - (X^\perp\mathcal D_{i, 0}) (f)  =-   \mathcal D_{i, 1} ( f_i ) \not =0\,.
 $$ 
 However, by \cite{Fo07} it follows that  $u_i \in W^t (M)$ for all $t < 1-\lambda_i$.  Hence there exists $f \in C^\infty_0(M\setminus \Sigma)$ such that 
the cohomological equation $\bar X u=f$ has its unique zero-average  solution $u \in W^t(M) \setminus W^1(M)$ for all $t <1-\lambda_g$. 
The statement is thus proved for $k=0$ with $\lambda := 1-\lambda_g >0$.

The statement for all integer $k >0$ can be proved by a similar argument. Let $\mathcal I^{k+1}_{\bar X}$ denote the space of ${\bar X}$-invariant distributions
of Sobolel order $\leq k$. By  \cite{Fo07} whenever $f \in C^\infty_0(M\setminus \Sigma)$ belongs to the joint kernel of the space $\mathcal I^{k+1}_{\bar X}$, the
solution $u$ of the cohomological equation ${\bar X} u=f$ belongs to the Sobolev space $W^{k} (M)$.
 
  Let then $f_{k,i} \in C_0^\infty(\setminus\Sigma)$ be a function in the kernel of $\mathcal I^{k+1}_{\bar X}$, so that
  as pointed above the cohomological equation $\bar X u=f$ has a continuous solution which belongs to $W^k(M)$, but with the property that there exists 
  $i \in \{2, \dots, g\}$  such that
  $$
  \mathcal D_{i, k+1} ( f_{k,i} ) \not =0 \quad \text{ and } \quad  \mathcal D_{a, h} ( f_{k,i} ) =0\,,  \quad \text{for }  (a, h) \not = (i, k+1)  \,.
  $$
  Under these hypotheses by \cite{Fo07} the cohomological equation $Xu_{k,i} = f_{k,i}$ has its unique zero-average solution $u_{k,i} \in W^t(M) \setminus W^{k+1}(M)$
  for  all $t < k+ 1-\lambda_i$.
    
  \end{proof} 
  
  \subsection{Heisenberg nilflows}  
  
  Let $\phi_\R$ be a nilflow on a Heisenberg nilmanifold $M= \Gamma \backslash H_3$, which is a quotient of the $3$-dimensional Heisenberg group
  $H_3$ by a (co-compact) lattice $\Gamma < H_3$.  A Heisenberg nilmanifold is Seifert space, in the sense that it is a circle fibration, over a 
  $2$-dimensional torus. A non-vertical nilflow on $M$ projects onto a linear flow on the $2$-torus. By the classical theory of nilflows, a Heisenberg 
  nilflow is uniquely ergodic if and only if its toral projection is, and so if and only if its toral projection is an irrational linear flow.
  
  A Heisenberg nilflow satisfies a Diophantine condition if and only if its toral projection is a linear toral flow which satisfies a Diophantine condition.
    
  \begin{lemma}
   \label{lemma:CE_Heis}
   For all Diophantine Heisenberg nilflows on a Heisenberg nilmanifold $M$ and for all  $f \in C^\infty(M)$,
  if $u\in L^2(M)$ is a solution the cohomological equation
 $$
 u\circ \phi^X_t - u = \int_0^t  f \circ \phi^X_s ds   \,, \quad \text{ for all } t \in \R\,,
 $$
then $u \in C^\infty(M)$. 
  \end{lemma} 
  
 \begin{proof} By the theory of unitary representations of Heisenberg nilflows, the cohomological equation can be split into a countable number
 of equations for the orthogonal projections of the function $f$ onto the irreducible components of the space $L^2(M)$. 
 
 For the (one-dimensional)  irreducible components, their sum is the subspace of $L^2(M)$ of functions which are pull-back of toral functions, hence the cohomological equation
 for the Heisenberg nilflow reduces to the cohomological equation for a Diophantine linear flow, which has a solution $\bar u \in C^\infty(\T^2)$ by the 
 Diophantine condition.  
 
 All the remaining irreducible unitary representations are infinite dimensional. If the cohomological  equation has a continuous, hence square-integrable
 solution, it follows that $f$ belongs to the kernel of all $X$-invariant distributions. In fact, let $H_\lambda \subset L^2(M)$ denote an irreducible component
 of central parameter $\lambda \in \R$. It is well known that irreducible unitary representations are classified by their central parameters (by the Stone-Von Neumann 
 theorem)  and that for each central parameter the corresponding irreducible representation $H_\lambda$ appears with finite multiplicity as a direct summand of 
 $H_\lambda$.

By orthogonal projection of the cohomological equation on $H_\lambda$ we derive that the cohomological equation
$$
 u_\lambda \circ \phi^X_t - u_\lambda = \int_0^t  f_\lambda \circ \phi^X_s ds   \,, \quad \text{ for all } t \in \R\,,
 $$
 for the orthogonal projection $f_\lambda \in H_\lambda \cap C^\infty(M)$ of the function $f$ onto $H_\lambda$, has a solution
 $u_\lambda \in L^2(M)$, hence in particular
 $$
 \Vert \int_0^t  f_\lambda \circ \phi^X_s ds  \Vert_{L^2(M)}  \leq   2 \Vert u_\lambda \Vert_ {L^2(M)} \quad \text{ for all } t \in \R \,.
 $$
 It follows that $f_\lambda$ belongs to the kernel of the (one-dimensional) space of invariant distributions supported on $H_\lambda$.
 In fact, there exists an appropriately normalized generator $\mathcal D_\lambda$ of the space of $X$-invariant distributions supported on
 $H_\lambda$ such that (see for instance \cite{FoKa20a}, Cor. 7.2) 
 $$
 \Vert  \frac{1}{\sqrt{t}} \int_0^t  f_\lambda \circ \phi^X_s ds  \Vert_{L^2(M)}   \to   \vert D_\lambda (f_\lambda) \vert\,.
 $$
 We have thus proved that, under the hypothesis that the cohomological equation for $f \in C^\infty(M)$ (in fact, $f \in W^r(M)$ with
 $r>1/2$ is sufficient) has a  solution $u\in L^2(M)$, then 
 $$
 D(f) =0  \,, \quad \text{ for  all  $X$-invariant distributions} \,.
 $$
 Finally, it follows from the results of \cite{FlaFo06} on the cohomological equation for Heisenberg nilflows that under the above hypothesis that
 $f \in C^\infty(M)$ belongs to the kernel of all invariant distributions, the solution $u\in L^2(M)$, unique up to additive constants, of the cohomological
 equation in fact belongs to $C^\infty(M)$, as stated.
 \end{proof}  
 
 \section{Main results}
 \label{sec:proofs}
 
 \subsection{Time changes}
 
 We give below more precise statements and proofs of Theorems \ref{thm:main1a} and  \ref{thm:main2a}  for time changes.
 
  \begin{theorem}  Let $\phi^f_\R$ and $\phi^g_\R$ be time-changes of a horocycle flow or of a generic Kochergin flow. 
 The set of pairs of functions $(f, g)$ such that the flows $\phi^f_\R$  and $\phi^g_\R$ are $C^1$ conjugate is a  positive, but 
 finite, codimension subspace of the space of pair of  functions $(f,g)\in C^\infty(M)^2$ such that $\phi^f_\R$ and $\phi^g_\R$ are  $C^0$  conjugate. 
 
 In addition, for all
 $k\in \N$, the set of of pairs of functions $(f, g)$  such that the conjugacy between the flows $\phi^f_\R$ and $\phi^g_\R$  is 
 of class $C^0 \cap  W^{k+1}$  is a  positive, but finite, codimension subspace of the space of  pair functions $(f,g)\in C^\infty(M)^2$ such that the conjugacy 
 between the flows $\phi^f_\R$ and $\phi^g_\R$  is  of class $C^0 \cap W^{k}$.
 
 For horocycle flows, and for any given $T>0$,  an analogous result holds for the  time-$T$ maps $\phi^f_T$ and $\phi^g_T$ of the time-changes 
 $\phi^f_\R$ and $\phi^g_\R$. 
  \end{theorem} 
 \begin{proof} 
 For horocycle flows, Ratner's rigidity theorem  \ref{thm:Ratner_rig} states that  if the time-changes $\phi^f_\R$ and $\phi^g_\R$ are measurably isomorphic, or equivalently if for a given $T>0$ the time-$T$ maps $\phi^f_T$ and $\phi^g_T$ are measurably isomorphic, then  there exists 
 an automorphism $\psi:M\to M$ given by an element of the normalizer of the lattice, such that $f\circ \psi -g$ is a measurable coboundary. By Lemma \ref{lemma:cocycle_rigidity}
$f\circ \psi -g$ is a continuous coboundary,  and by Lemma~\ref{lemma:CE_Horo} there exist smooth functions $f, g\in C^\infty(M)$ such that the primitive  of $f -g$, hence the primitive of $f\circ \psi-g$,  is continuous, but does not belong to $W^1$.
 
 \smallskip
 For Kochergin flows, if the time changes $\phi^f_\R$ and $\phi^g_\R$ are $C^0$ conjugate, then the conjugating homeomorphism $h:M\to M$ maps the
 orbit foliation onto itself, hence by Lemma~\ref{lemma:Koch_centr} we can assume that $h$ is isotopic to the identity, hence it fixes the singularities, the
 singular leaves, and so all leaves of the orbit foliation.  It follows that if $\phi^f_\R$ and $\phi^g_\R$ are $C^1$ conjugate, then in particular 
 $f-g$ is a $C^1$ coboundary, but by Lemma~\ref{lemma:CE_Koch} there exist smooth functions $f, g\in C^\infty(M)$ such that the primitive  of $f -g$ is continuous, but 
 does not belong to $W^1$.
 
\smallskip
 By a similar argument, we can derive from \ref{lemma:CE_Horo}  and Lemma~\ref{lemma:CE_Koch}  that, for all $k\in \N$ there exist pairs $(f, g)$
  of smooth functions such that $\Phi^f_\R$ and $\Phi^g_\R$ are $C^0$ conjugate, and the conjugacy belongs to the Sobolev space $W^k$, but not to
  its subspace $W^{k+1}$. In fact, the space of smooth pairs $(f,g)$ such that the conjugacy is of class $W^{k+1}$ has positive, but finite, codimension
  in the space of pairs for which skew-product flows are $C^0$ conjugated and the conjugacy is of class $W^k$. 
 
 \end{proof} 
 
  \begin{theorem} 
  Let $\phi^f_\R$  be smooth time changes of a given Diophantine Heisenberg nilflow $\phi_\R$. If the flows $\phi^f_\R$ and $\phi_\R$ are $C^0$ conjugate, then 
  they are $C^\infty$ conjugate.
   \end{theorem} 
  \begin{proof} 
 Since the flows $\phi^f_\R$ and $\phi_\R$ are topologically conjugate, it follows that the time-change $\phi^f_\R$ is not weakly mixing, hence by  Lemma
 the time-change is trivial, and in fact $f$  is smoothly cohomologous to a constant, hence $\phi^f_\R$ is smoothly conjugate to $\phi_\R$.
   \end{proof}
 
 We remark that we don't know whether  if non-trivial time changes $\phi^f_\R$ and $\phi^g_\R$ are $C^0$ conjugate, then $f-g$ is a coboundary,
 hence we have no rigidity statement, for time changes of Heisenberg nilflows, even within the class of time changes.  This motivates the following question:
 \begin{question}Assume that for any pair of smooth functions  $f,g$ on $M$, the corresponding time changes $\phi^f_\R$ and $\phi^g_\R$ of a (Diophantine) Heisenberg 
 nilflow  $\phi_\R$ are $C^{0}$ conjugate. Is $f-g$ a (smooth) coboundary?
 \end{question}

 \subsection{Skew-product flows}
Similar statements hold for skew-product flows and make precise Theorems \ref{thm:main1b} and  \ref{thm:main2b}.

 \begin{theorem}  Let $\Phi^f_\R$ and $\Phi^g_\R$ be skew-product flows over the horocycle flows or a generic Kochergin flow. 
 The set of pairs of functions $(f, g)$ such that the flows $\Phi^f_\R$  and $\Phi^g_\R$ are $C^1$ conjugate is a  positive, but 
 finite, codimension subspace of the space of pair of  functions $(f,g)\in C^\infty(M)^2$ such that $\Phi^f_\R$ and $\Phi^g_\R$ are $C^0$ conjugate. 
 
 In addition, for all
 $k\in \N$, the set of of pairs of functions $(f, g)$  such that the conjugacy between the flows $\Phi^f_\R$ and $\Phi^g_\R$  is 
 of class $C^0 \cap  W^{k+1}$  is a  positive, but finite, codimension subspace of the space of  pair functions $(f,g)\in C^\infty(M)^2$ such that the conjugacy 
 between the flows $\Phi^f_\R$ and $\Phi^g_\R$  is  of class $C^0 \cap W^{k}$.
  \end{theorem} 
  \begin{proof}
  By Lemma \ref{lemma:CE_Horo}  and Lemma~\ref{lemma:CE_Koch} on the cohomological equations for the horocycle flows and Kochergin flows,
  respectively, there exists $f, g \in C^\infty(M)$, and $f,g \in C^\infty_0(M\setminus \Sigma)$ in the Kochergin case, such that the cohomological equation 
  for $f-g$   has a $C^0$ solution, but no $W^1$ solution. Indeed, it follows from the argument that the pairs $(f,g)$ such that $f-g$ is a $C^0$ coboundary
  has countable codimension (vanishing of all horocycle invariant distributions from the principal and complementary series, and from discrete series
  components of parameter $n=1$) for the horocycle, and finite codimension (equal to the genus) for Kochergin flows.   By Lemma \ref{lemma:conj_skew_shifts1}
  for such pairs $(f,g)$ the skew-product flows $\Phi^f_\R$ and $\Phi^g_\R$ are $C^0$ conjugate.
  
  Conversely,  since horocycle flows and Kochergin flows are expansive,  it follows from~\ref{lemma:conj_skew_shifts2} that for any $C^0$ conjugacy of
  the flows  $\Phi^f_\R$ and $\Phi^g_\R$ there exists an automorphism $\bar h:M\to M$ of the flow $\phi_\R$ such that the conjugacy comes from a $C^0$ 
  transfer function for $f \circ \bar h-g$.  By Ratner's rigidity theorem \cite{Rat82}, Cor. 2,  the essential centralizer of the horocycle flow coincides with the normalizer
  of the lattice $\Gamma$ such that $M := \Gamma \backslash PSL(2, \R)$. Since the action of the normalizer preserves the representation types and the 
  space of horocycle invariant distributions, we can assume that the normalizer is trivial.  Similarly, by Lemma \ref{lemma:Koch_centr}
  a generic Kochergin flow has a trivial essential centralizer.  It follows that the smooth functions $f \circ \bar h-g$ and $f-g$ differ by a smooth coboundary, 
  hence  $f \circ \bar h-g$ and $f-g$ has primitives of exactly the same regularity.
  
  Since there exists pairs $(f,g)$ such that $f-g$ has a $C^0$
  transfer function which is not $W^1$ (hence not $C^1$), it follows that there exist skew-product flows which are $C^0$, but not $C^1$ conjugate.   In fact,
  the set of pairs $(f,g)$ such that $\Phi^f_\R$ and $\Phi^g_\R$ are $C^1$ conjugate has positive, but finite, codimension within the space of pairs such that 
  $\Phi^f_\R$ and $\Phi^g_\R$ are $C^0$ conjugate. 
  
  By a similar argument, we can derive from \ref{lemma:CE_Horo}  and Lemma~\ref{lemma:CE_Koch}  that, for all $k\in \N$ there exist pairs $(f, g)$
  of smooth functions such that $\Phi^f_\R$ and $\Phi^g_\R$ are $C^0$ conjugate, and the conjugacy belongs to the Sobolev space $W^k$, but not to
  its subspace $W^{k+1}$. In fact, the space of smooth pairs $(f,g)$ such that the conjugacy is of class $W^{k+1}$ has positive, but finite, codimension
  in the space of pairs for which skew-product flows are $C^0$ conjugated and the conjugacy is of class $W^k$. 
 \end{proof}
 
 For skew-products over Heisenberg nilflows, we have the following result 
 
  \begin{theorem} 
  Let $\Phi^f_\R$ and $\Phi^g_\R$ be smooth skew-product flows over a given Diophantine Heisenberg nilflow, defined by a cocycle with values in the circle. If the flows $\Phi^f_\R$ and 
  $\Phi^g_\R$ are $C^0$ conjugate, then they are $C^\infty$ conjugate.
   \end{theorem} 
  \begin{proof} 
  Since any Heisenberg nilflow is expansive, by Lemma~\ref{lemma:conj_skew_shifts2} it follows that if $\Phi^f_\R$ and $\Phi^g_\R$  are $C^0$- conjugate, then there exists
  a $C^0$ automorphism $\bar h:M\to M$ of the Heisenberg nilflows such that the function $f\circ \bar h-g$ is a $C^0$ coboundary.  By Theorem \ref{thm:Heis_centr}, the automorphism $\bar h$ is given by the normalizer of the lattice, hence it is smooth and preserves the invariant distributions of the Heisenberg nilflow.
  By Lemma\ref{lemma:CE_Heis}  on the cohomological equation for Diophantine Heisenberg nilflows, if $f \circ \bar h -g$ is a smooth function which is a 
  $C^0$ coboundary, then it is a $C^\infty$ coboundary, that is, the primitive $u\in C^\infty(M)$. By ~\ref{lemma:conj_skew_shifts1} it follows that $\Phi^f_\R$ and 
  $\Phi^g_\R$ are $C^\infty$- conjugate.
  \end{proof}


\begin{thebibliography}{20}
 
\bibitem[Ar61]{Ar61}   V.~I.~Arnol'd, Small Denominators. I. Mapping of the Circle onto itself,  {\it Izv. Akad. Nauk SSSR Ser. Mat.} {\bf  25}:1 (1961),  21--86.  
Translations Amer. Math. Soc., 2nd series {\bf 46}, 213--284.
\bibitem[AFRU21]{AFRU21}  A. Avila, G. Forni, D. Ravotti and C. Ulcigrai, Mixing for Smooth Time-Changes of  General Nilflows, {\it Advances in Math.}  {\bf 385} (2021),
107759.
\bibitem[BuFo14]{BuFo14}  A. Bufetov and G. Forni, Limit theorems for horocycle flows,  {\it Ann. Sci. \'{E}c. Norm. Sup\'{e}r. (4)}  {\bf 47} (5) (2014), 851--903.
\bibitem[DFS]{DFS} D.~Damjanovic, B.~Fayad and M.~Saprykina,  On KAM rigidity of parabolic affine actions on the torus, talk by B. Fayad at the conference {\it Dynamische Systeme}, Oberwolfach, Germany, July 13, 2021.
\bibitem[DS17]{DS17}   D.~Dolgopyat  and O.~Sarig, Temporal distributional limit theorems for dynamical systems,
 {\it J. Stat. Phys.} (3-4) {\bf 166} (2017), 680--713.     
  \bibitem[FlaFo03]{FlaFo03} L.~Flaminio and G. Forni, Invariant distributions and time-averages for horocycle flows, {\it Duke Math. J.}~\textbf{119} (2003), 465--526.
 \bibitem[FlaFo06]{FlaFo06} \bysame, Equidistribution of nilflows and applications to theta sums, 
{\it Ergodic Theory Dynam. Systems}~\textbf{26} (2) (2006), 409--433.
\bibitem[FlaFo07]{FlaFo07} \bysame,  The cohomological equation for nilflows,  {\it J.
Mod. Dynam.} {\bf 1} (2007), 37--60.
\bibitem[FlaFo14]{FlaFo14} \bysame, On effective equidistribution for higher step nilflows, preprint,  arXiv:1407.3640v1.
\bibitem[FlaFo19]{FlaFo19}  \bysame, Orthogonal powers and M\"obius conjecture for smooth time changes of horocycle flows,
{\it Electronic Research Announcements In Mathematical Sciences}  \textbf{26} (2019), 16--23.
\bibitem[FFRH16]{FFRH16}  L. Flaminio, G. Forni and F.~Rodriguez Hertz,  Invariant distributions for homogeneous flows and 
affine transformations, {\it J. Mod. Dynam.}, {\bf 10} (2016), 33 --79.
\bibitem[FFT16]{FFT16}  L.~Flaminio, G.~Forni and J.~Tanis, Effective equidistribution of twisted horocycle flows and horocycle maps, {\it Geometric and Functional Analysis} \textbf{26 (5)} (2016),1359--1448.
\bibitem[Fo97]{Fo97} G.~Forni, Solutions of the Cohomological Equation for Area-Preserving Flows on Compact Surfaces of Higher Genus,
 {\it Ann. of Math.} \textbf{146}(2) (1997), 295--344. 
 \bibitem[Fo02]{Fo02}
\bysame, Deviation of ergodic averages for area-preserving flows on
  surfaces of higher genus, {\it Ann. of Math.} (2) \textbf{155} (2002), no.~1,
  1--103.
 \bibitem[Fo07]{Fo07} \bysame, Sobolev regularity of solutions of the cohomological equation,  {\it Erg. Th. Dynam. Sys.},  online at https://doi.org/10.1017/etds.2019.108
 (arXiv:0707.0940v2).
 \bibitem[Fo08]{Fo08} \bysame, On the Greenfield-Wallach and Katok conjectures, in {\it Geometric and Probabilistic 
Structures in Dynamics} (K. Burns, D. Dolgopyat and Ya. Pesin Editors). Contemporary Mathematics (Proceedings), pp. 197-215. American Mathematical Society, Providence RI, 2008.
 \bibitem[FoKa20a]{FoKa20a}  G. Forni and A. Kanigowski, Time-changes of Heisenberg nilflows,  {\it Ast\'erisque}  \textbf{416} (2020), 253--299.
  \bibitem[FoKa20b]{FoKa20b}  \bysame, Multiple mixing and disjointness for time changes of bounded-type Heisenberg nilflows, 
{\it J.  \'Ecole polytechnique -- Math\'ematiques} {\bf 7}  (2020) , 63--91.
 \bibitem[Kat03]{Kat03} A.~Katok, Combinatorial constructions in ergodic theory and dynamics, University Lecture Series, vol. \textbf{30}, American 
Mathematical Society, Providence, RI, 2003. 
\bibitem[KL17]{KL17}  A.~Kanigowski and  M.~Lemanczyk,  Flows with Ratner's property have discrete essential centralizer,  {\it Studia Mathematica} {\bf 237} (2017), 185--194
\bibitem[KLU20]{KLU20} 
 A.~Kanigowski, M.~Lemanczyk and C.~Ulcigrai,  
On disjointness properties of some parabolic flows, {\it Inventiones Mathematicae} (2020)  \textbf{221}, 1--111.
\bibitem[KMS91]{KMS91} H.~Keynes, N.~Markley and M. Sears,  The structure of automorphisms of real suspension flows,
{\it Erg. Th. Dynam. Sys.}  {\bf 11} (2) (1991), 349--364. 
 \bibitem[Kha18]{Kha18} K.~Khanin, Proc. Int. Cong. of Math.,  2018 Rio de Janeiro, Vol. {\bf 3} (1991--2012)
 \bibitem[Ko09]{Ko09} A.~Kocsard,  Cohomologically rigid vector fields: the Katok conjecture in dimension $3$,  Ann. I. H. Poincar\'e  {\bf 26} (2009), 1165--1182.
 \bibitem[MMY05]{MMY05} S. Marmi, P. Moussa and J.-C. Yoccoz, The cohomological equation for Roth-type
interval exchange maps, {\it J. of the AMS}  \textbf{18} (4) (2005), 823--872.
 \bibitem[MMY12]{MMY12}
\bysame, Linearization of generalized interval exchange maps, 
{\it Ann. of Math.} \textbf{176}, no. 03, (2012), 1583--1646.
   \bibitem[MY16]{MY16} S.~Marmi and J.-C. Yoccoz, H\"older regularity of the solutions of the cohomological equation 
  for Roth type interval exchange maps, {\it Comm.  Math. Phys.} \textbf{344}  (1) (2016), 117--139.
 \bibitem[Ma09]{Ma09} S.Matsumoto, The parameter rigid flows on oriented 3-manifolds, {\it Contemporary Mathematics}  {\bf 498}. Foliations, Geometry, and Topology: Paul Schweitzer Festschrift  (ed. N. C. Saldanha et al.)  American Math. Soc. 2009,  135--139.
\bibitem[Rat82]{Rat82} M. Ratner, Rigidity of Horocycle Flows,   {\it Annals of Mathematics} {\bf 115} (3) (1982), 597--614. 
\bibitem[Rat87]{Rat87} \bysame, Rigid reparametrizations and cohomology for horocycle flows, {\it Inventiones Math.}
{\bf 88}(2)  (1987), 341--374.
 \end{thebibliography}
 \end{document}